\documentclass[12pt, twoside, leqno]{article}

\usepackage{amsmath,amsthm}
\usepackage{amssymb}

\usepackage{enumerate}

\usepackage{graphicx}

\newtheorem{thm}{Theorem}[section]

\newtheorem{lem}[thm]{Lemma}
\newtheorem{prop}[thm]{Proposition}
\newtheorem{prob}[thm]{Problem}

\theoremstyle{definition}
\newtheorem{defin}[thm]{Definition}
\newtheorem{rem}[thm]{Remark}

\numberwithin{equation}{section}

\newcommand{\eps}{\varepsilon}

\begin{document}


\baselineskip=17pt


\title{Universal stability  of Banach spaces for $\eps$-isometries}

\author{Lixin Cheng$^\dag$\\
School of Mathematical Sciences\\
Xiamen University\\
Xiamen 361005, China\\
E-mail: lxcheng@xmu.edu.cn
\and
Duanxu Dai$^\ddag$\\
School of Mathematical Sciences\\
Xiamen University\\
Xiamen 361005, China\\
E-mail: dduanxu@163.com
\and
Yunbai Dong$^\S$\\
School of Mathematics and Computer\\
Wuhan Textile University\\
Wuhan 430073, China\\
E-mail: baiyunmu301@126.com
\and
Yu Zhou\\
School of Fundamental Studies\\
Shanghai University of Engineering Science\\
Shanghai 201620, China\\
E-mail: Roczhou\_fly@126.com}

\date{}

\maketitle


\renewcommand{\thefootnote}{}

\footnote{2010 \emph{Mathematics Subject Classification}: Primary 46B04, 46B20, 47A58; Secondary 26E25, 46A20, 46A24.}

\footnote{\emph{Key words and phrases}: $\eps$-isometry, stability, injective space, Banach space.}

\footnote{$^\dag$ Support in partial
by NSFC, grants 11071201 \& 11371296.}
\footnote{$^\ddag$ Support in partial  by a fund of China Scholarship Council and Texas A\&M University.}
\footnote{$^\S$ Support in partial
by NSFC, grant 11201353.}

\renewcommand{\thefootnote}{\arabic{footnote}}
\setcounter{footnote}{0}


\begin{abstract}
Let $X$, $Y$ be two real Banach spaces and $\eps>0$.
A standard  $\eps$-isometry $f:X\rightarrow Y$ is said to be $(\alpha,\gamma)$-stable (with respect to $T:L(f)\equiv\overline{{\rm span}}f(X)\rightarrow X$ for some $\alpha, \gamma>0$) if $T$ is a linear operator with $\|T\|\leq\alpha$ so that $Tf-Id$ is uniformly bounded by $\gamma\eps$ on $X$.
The pair $(X,Y)$  is said to be stable  if
   every standard $\eps$-isometry $f:X\rightarrow Y$ is $(\alpha,\gamma)$-stable for some $\alpha,\gamma>0$. $X (Y)$ is said to be universally left (right)-stable, if $(X,Y)$ is always stable for every $Y (X)$. In this paper, we show  that  universal right-stability spaces are just Hilbert spaces; every injective space is universally left-stable;   a Banach space $X$ isomorphic to a subspace of $\ell_\infty$ is universally left-stable if and only if it is isomorphic to $\ell_\infty$; and that  a separable space $X$ satisfies the condition that $(X,Y)$ is  left-stable for every separable $Y$ if and only if it is  isomorphic to $c_0$.
\end{abstract}

\section{Introduction}

The study of properties of isometries and its generalizations between Banach spaces has continued for 80 years since Mazur and Ulam 1932's celebrated results \cite{ma}: Every surjective isometry between two Banach spaces $X$ and $Y$ is necessarily affine. While a simple example $f:\mathbb R\rightarrow\ell^2_\infty$ defined by $f(t)=(t,\sin t)$ shows that it is not always possible if the mapping is not surjective. In 1968, Figiel \cite{figiel} showed the following remarkable result: For every standard isometry $f:X\rightarrow Y$  there is a linear operator $T:L(f)\rightarrow X$ with $\|T\|=1$ so that $Tf=Id$ on $X$, where $L(f)$ is the closure of span$f(X)$ in $Y$ (see, also, \cite{ben} and \cite{du}). In 2003, Godefroy and Kalton \cite{gode} studied the relationship between isometries and linear isometries and resolved a long-standing problem: Whether existence of an isometry $f:X\rightarrow Y$ implies the existence of a linear isometry $U:X\rightarrow Y$?

\begin{defin} Let $X,Y$ be two Banach spaces, $\eps\geq0$, and let $f:X\rightarrow Y$ be a mapping.

(1) $f$ is said to be an $\eps$-isometry if
\begin{align} |\|f(x)-f(y)\|-\|x-y\||\leq\eps\;\;{\rm for\; all}\; x,y\in X.
\end{align}
In particular, a $0$-isometry $f$ is simply called an isometry.

(2) We say  an $\eps$-isometry $f$ is standard if $f(0)=0$.

(3) A standard $\eps$-isometry is $(\alpha,\gamma)$-stable if there exist $\alpha, \gamma>0$ and a bounded linear operator $T:L(f)\rightarrow X$ with $\|T\|\leq\alpha$ such that
\begin{align}\label{E:1.2}\|Tf(x)-x\|\leq\gamma\eps,\;\;{\rm for\;all\;}x\in X.
\end{align}
In this case, we also simply say $f$ is stable, if no confusion arises.

(4) A pair $(X,Y)$ of Banach spaces $X$ and $Y$ is said to be stable  if
   every standard $\eps$-isometry $f:X\rightarrow Y$ is $(\alpha,\gamma)$-stable for some $\alpha,\gamma>0$.

(5) A pair $(X,Y)$ of Banach spaces $X$ and $Y$  is called  $(\alpha,\gamma)$-stable for some $\alpha,\gamma>0$ if every standard $\eps$-isometry $f:X\rightarrow Y$ is $(\alpha,\gamma)$-stable.
\end{defin}

In 1945,  Hyers and Ulam  proposed the following question\cite{hy} (see, also \cite{om}): Whether
for every every pair of Banach spaces $(X,Y)$ there is $\gamma>0$ such that for every standard surjective $\eps$-isometry $f:X\rightarrow Y$ there exists a surjective linear isometry $U:X\rightarrow Y$
 so that
\begin{align}\label{E:1.3}\|f(x)-Ux\|\leq\gamma\varepsilon,~~~~{\rm for~~all}\; x\in X .\end{align}

After many years efforts of a number of mathematicians (see, for instance, \cite{ge}, \cite{gru}, \cite{hy}, and  \cite{om}),  Omladi\v{c} and \v{S}emrl \cite{om} finally achieved the sharp estimate $\gamma=2$ in \eqref{E:1.3}.

The study of non-surjective  $\eps$-isometries has also brought to mathematicians' attention (see, for instance, \cite{bao}, \cite{cheng}, \cite{cheng1},  \cite{dil}, \cite{om}, \cite{qian},  \cite{sm} and \cite{ta}).
Qian\cite{qian} first proposed  the following problem in 1995.

\begin{prob} Whether for every pair $(X,Y)$ of Banach spaces $X$ and $Y$ there exists  $\gamma>0$ such that every standard
$\varepsilon$-isometry $f:X\rightarrow Y$ is $(\alpha,\gamma)$-stable for some $\alpha>0$.
\end{prob}
\noindent
Then he showed that the answer is affirmative if both $X$ and $Y$ are
$L_p$ spaces. \v{S}emrl and V\"{a}is\"{a}l\"{a} \cite{sm} further presented a sharp estimate of \eqref{E:1.2}
with $\gamma=2$ if both $X$ and $Y$ are $L^p$ spaces for $1<p<\infty$.
However,  Qian  \cite{qian} presented a counterexample showing that if a separable Banach space $Y$ contains a uncomplemented closed subspace $X$ then for every $\eps>0$ there is a standard $\eps$-isometry $f:X\rightarrow Y$ which is not stable.  Cheng, Dong and Zhang \cite{cheng} showed the following weak stability version.
  \begin{thm}[Cheng-Dong-Zhang]\label{T:1.3}
 Let $X$ and $Y$ be Banach spaces,
  and let $f:X\rightarrow Y$ be a standard $\eps$-isometry for
 some $\eps\geq 0$. Then for every $x^*\in X^*$, there exists $\phi\in Y^*$ with $\|\phi\|=\|x^*\|\equiv r$ such that
 \begin{align} |\langle\phi,f(x)\rangle-\langle x^*,x\rangle|\leq4\eps r, \; for\;all\;x\in X.\end{align}
 \end{thm}

For study of the stability of $\eps$-isometries of Banach spaces, the following two questions are very natural.
\begin{prob}
Is there a  characterization for the class of Banach spaces $\mathcal{X}$ satisfying given any $X\in \mathcal{X}$ and  Banach space $Y$, the pair $(X,Y)$ is ($(\alpha,\gamma)$-, resp.) stable?

\end{prob}

Every space $X$ in this class is said to be a universal ($(\alpha,\gamma)$-, resp.) left-stability space.
\begin{prob}
 Can we characterize the class of Banach spaces $ \mathcal{Y}$, such that given any $Y\in \mathcal{Y}$ and Banach space $X$, the pair $(X,Y)$ is ($(\alpha,\gamma)$-, resp.) stable ?

\end{prob}
Every space $Y$ in this class is called a universal ($(\alpha,\gamma)$-, resp.) right-stability space.\\

In this paper, we study  universal stability and universal right-stability of Banach spaces. As a result, with the aim of  Qian's counterexample and Theorem \ref{T:1.3}, incorporating of  Lindenstrauss-Tzafriri's characterization of Hilbert spaces \cite{lin}, we show  that universal stability spaces are  spaces of finite dimensions; and  up to an isomorphism, a universal right-stability space is  just a Hilbert space. By using
 Theorem \ref{T:1.3}  we then prove that every injective space is universally  left-stable; and a Banach space $X$ which is isomorphic to a subspace of $\ell_\infty$ is universally left-stable if and only if it is isomorphic to $\ell_\infty$. Finally, applying  Zippin's theorem \cite{zip} we verify that a separable space $X$ satisfies that $(X,Y)$ is stable for every separable $Y$ if and only if it is  isomorphic to $c_0$.

All symbols and notations in this paper are standard. We use $X$ to denote a real Banach space and $X^*$ its dual. $B_X$ and $S_X$ denote the closed unit ball and the unit sphere of $X$, respectively.
Given a bounded linear operator $T:X\rightarrow Y$, $T^*:Y^*\rightarrow X^*$ stands for its conjugate operator. For a subset $A\subset X$, $\overline{A}$ stands for the closure of $A$, and card$(A)$, the cardinality of $A$.

\section{Universal (right-) stability spaces for $\eps$-isometries}
In this section, we search for some properties of the class of universal left (right)-stability spaces for $\eps$-isometries.

Recall that a Banach space $X\; (Y)$ is universally left (right)-stable if it satisfies that for every Banach space $Y \;(X)$ and for every standard $\eps$-isometry $f:X\rightarrow Y$,
there exist  $\alpha, \gamma>0$ and a bounded operator $T:$
$L(f)\rightarrow X$ with $\|T\|\leq\alpha$ so that
\begin{align}\label{E:2.1}
\|Tf(x)-x\|\leq \gamma\varepsilon ,\;\;for\;all \;x\in X.\end{align}
A universal stability space is a Banach space which is both universally left and right stable.  As a result,  we show inequality \eqref{E:2.1} holds for every Banach space $X$ if and only if $Y$ is, up to linear isomorphism, a Hilbert space;  and universal stability spaces are just finite dimensional spaces.

The following lemma follows from Qian's counterexample\cite{qian}.
\begin{lem}\label{L:2.1}
Let $X$ be a closed subspace of a Banach space $Y$. If $\text{card}(X)=\text{card}(Y)$, then  for every $\eps>0$ there is a standard $\eps$-isometry $f:X\rightarrow Y$
such that

(1)  $L(f)\equiv\overline{\text{span}}f(X)=Y$;

(2)  $X$ is complemented whenever $f$ is stable.
\end{lem}

\begin{thm}\label{T:2.2} Let $Y$ be a Banach space. Then the following statements are equivalent.

i) $Y$ is universally right-stable;

ii) $Y$ is isomorphic to a Hilbert space;

iii) $Y$ is universally $(\alpha, 4)$-right-stable for some $\alpha>0$.

\end{thm}
\begin{proof}
i) $\Longrightarrow$ ii). By definition of universal right-stability, every closed  subspace of $Y$ is again universally right-stable. Fix any closed separable subspace $Z$ of $Y$. By Lemma \ref{L:2.1}, universal right-stability of $Z$ entails that every closed subspace of $Z$ is complemented in $Z$. According to Lindenstrauss-Tzafriri's theorem \cite{lin}: "a Banach space satisfying that every closed subspace is complemented is isomorphic to a Hilbert space", $Z$ is isomorphic to a (separable) Hilbert space. Hence, $Y$ itself is isomorphic to a Hilbert space.

ii) $\Longrightarrow$ iii). Suppose that $Y$ is  isomorphic to a Hilbert space $H$. Let $\alpha=\text{dist}(Y,H)$, the Banach-Mazur distance between $Y$ and $H$.  Then every closed subspace of $Y$ is $\alpha$-complemented in $Y$. Given $\eps>0$ and any standard  $\eps$-isometry $f:X\rightarrow Y$,  according to Theorem 4.8 of \cite{cheng}, inequality \eqref{E:2.1} holds for some $T: L(f)\rightarrow X$ with $\|T\|\leq\alpha$ and with $\gamma=4$, i.e., $Y$ is universally $(\alpha,4)$-right stable.

iii) $\Longrightarrow$ i). It is trivial.
\end{proof}

\begin{thm}\label{T:2.3}
A normed space $X$ is universally-stable if and only if it is finite dimensional.
\end{thm}
\begin{proof}
Sufficiency.
Since every finite dimensional normed space is isomorphic to an Euclidean space, Theorem \ref{T:2.2} entails that it is universally right-stable. While Theorem 3.4 of \cite{cheng} says that $n$ dimensional spaces are universally left-stable with the parameter $\gamma=4n$.

Necessity. Suppose, to the contrary, that $X$ is infinite dimensional. Since $X$ is also universally right-stable, according to Theorem \ref{T:2.2} we have just proven, it is isomorphic to a Hilbert space. Since every closed subspace of a universally right-stable space is again universally right-stable, we can assume that $X$ is separable. Thus, $X$ is isometric to a  subspace of $\ell_\infty$. Since $\ell_\infty$ is prime \cite{lind} (i.e. every complemented infinite dimensional subspace is isomorphic to it), $X$ is uncomplemented in $\ell_\infty$.  Note ${\rm card}(X)={\rm card}\ell_\infty$. By Lemma \ref{L:2.1}, there is a unstable standard $\eps$-isometry $f:X\rightarrow\ell_\infty$ for every $\eps>0$, which is a contradiction to universal stability of $X$.
\end{proof}

\section{Universal left-stability spaces}

In this section, we consider properties of universal left-stability spaces. We shall show that (1) an injective Banach space is universally left-stable; (2)  a Banach space  isomorphic to a subspace of $\ell_\infty$ is universally left-stable if and only if it is  isomorphic to $\ell_\infty$ and (3) for a separable Banach space $X$, $(X,Y)$ is stable for every separable Banach space $Y$ if and only if $X$ is a separably injective Banach space.

A Banach space $X$ is said to be $\lambda$-injective (or, simply, injective) if it has the following extension property:
Every bounded linear operator $T$ from a closed subspace of a Banach space into $X$ can be extended to be a bounded operator on the whole space with its norm at most $\lambda\|T\|$ (see, for instance,  \cite{Alb}).  Goondner \cite{Goo} introduced a family of Banach spaces coinciding with the family of injective spaces: for any $\lambda\geq1$, a Banach space X is a $P_\lambda$-space if, whenever $X$ is
isometrically embedded in another Banach space, there is a projection onto the
image of X with norm not larger than $\lambda$. The following result is due to Day \cite{day} (see, also,  Wolfe \cite{w}, Fabian et al. \cite{fa}, p. 242).

\begin{prop}\label{P:3.1} A Banach space $X$ is $\lambda$-injective
if and only if it is a $P_\lambda$-space.
\end{prop}

\begin{rem}\label{R:3.2}
For any set $\Gamma$, that  $\ell_\infty(\Gamma)$ is  $1$-injective follows from the Hahn-Banach theorem.
\end{rem}
\begin{thm}\label{T:3.3} Every $\lambda$-injective space is universally $(\lambda,4\lambda)$-left-stable.
\end{thm}
\begin{proof}  Let $X$ be a $\lambda$-injective Banach space. We can assume that $X$ is a closed complemented subspace of $\ell_\infty(\Gamma)$; otherwise, we can identify $X$ for its canonical embedding  $J(X)$ as a  subspace of
$\ell_\infty(\Gamma)$, where $\Gamma$ denotes the closed ball $B_{X^*}$ of $X^*$. Hence, it is $\lambda$-complemented in $\ell_\infty(\Gamma)$. Let $P: \ell_\infty(\Gamma)\rightarrow X$ be a projection such that $\|P\|\leq\lambda$. Given any $\beta\in \Gamma$, let $\delta_\beta \in\ell_\infty(\Gamma)^*$ be defined for $x=(x(\gamma))_{\gamma \in
\Gamma}\in \ell_\infty(\Gamma)$ by $\delta_\beta (x)=x(\beta).$ Assume
that $f : X\rightarrow Y$ is a standard $\eps$-isometry. For every $x^* \in X^*$,
by Theorem \ref{T:1.3}, there is
$\phi\in Y^*$ with $\|\phi\|=\|x^*\|$ such that
\begin{align}\label{E:3.1}
|\langle\phi,f(x)\rangle-\langle x^*,x\rangle|
\leq4\eps\|x^*\|,\; {\rm for \;all}\; x\in X.\end{align}
In particular, letting $x^*=\delta_\gamma$ in \eqref{E:3.1} for every fixed $\gamma\in\Gamma$, we obtain
a linear functional $\phi_\gamma\in Y^*$ satisfying \eqref{E:3.1} with $\|\phi_\gamma\|=\|\delta_\gamma\|_X\leq1.$ Therefore, $(\phi_\gamma(y))_{\gamma\in\Gamma}\in\ell_\infty(\Gamma)$ for every $y\in Y$.

Let $T(y)=P(\phi_{\gamma}(y))_{\gamma \in \Gamma}$, for all $y\in Y$, and note $P|_X=I_X$, the identity from $X$ to itself.
Then $\|T\|\leq\|P\|\leq\lambda$ and for all $x\in X$,

\begin{align}\|Tf(x)-x\|&=\|P(\phi_{\gamma}(f(x)))_{\gamma \in
\Gamma}-(\delta_{\gamma}(x))_{\gamma \in \Gamma}\|\notag\\
&=\|P(\phi_{\gamma}(f(x)))_{\gamma \in
\Gamma}-P((\delta_{\gamma}(x))_{\gamma \in \Gamma})\|\notag\\
&\leq\|P\|\cdot\|(\phi_{\gamma}(f(x)))_{\gamma \in
\Gamma}-(\delta_\gamma(x))_{\gamma \in \Gamma}\|_\infty\leq 4\lambda\eps.\notag\end{align}
\end{proof}
\begin{thm}\label{T:3.4}
 Let $X$ be a Banach space $X$ isomorphic to an infinite dimensional subspace of $\ell_\infty$. Then the following statements are equivalent.

 i) $X$ is universally left-stable;

 ii) $X$ is  isomorphic to $\ell_\infty$.

 iii) $X$ is universally $(\lambda,4\lambda)$-left stable, where $\lambda=\text{dist}(X,\ell_\infty)$.
\end{thm}
\begin{proof}

i) $\Longrightarrow$ ii). Since $\dim X=\infty$ and since it is isomorphic to a subspace of $\ell_\infty$, we have
\begin{align}\text{card}(X)\geq\aleph=\aleph^{\mathbb N}=\text{card}(\mathbb R^\mathbb N)=\text{card}(\ell_\infty)\geq \text{card}(X).\end{align}
 Assume that $X$ is universally left-stable.
We can put an equivalent norm $|\|\cdot|\|$ on $\ell_\infty$ such that $X$ is isometric to a closed subspace of $(\ell_\infty,|\|\cdot|\|)$. Indeed, Let $T: X\rightarrow\ell_\infty$ be a linear embedding and let $|\cdot|$ on $Z\equiv T(X)$ be defined by
$|z|=\|x\|$ for all $z=Tx\in Z$. Then, we choose a sufficiently large $\lambda>0$ and define $\||\cdot\||$ on $\ell_\infty$ by
$\||u\||=\inf\{|v|+\lambda\|u-v\|: v\in Z\}.$ Clearly, the norm $\||\cdot\||$ has the property we desired. Applying Lemma \ref{L:2.1}, we observe that $X$ is complemented in  $(\ell_\infty,\||\cdot|\|)$, hence, in $\ell_\infty.$ By Lindenstrauss' theorem \cite{lind},  $X$ is isomorphic to $\ell_\infty$.

ii) $\Longrightarrow$ iii). Suppose that $X$ is isomorphic to $\ell_\infty$. Since $\ell_\infty$ is $1$-injective (Remark \ref{R:3.2}), $X$ is necessarily $\lambda$-injective ($\lambda=\text{dist}(X,\ell_\infty)$). By Theorem 3.3, $X$ is universally $(\lambda,4\lambda)$-left stable.

iii) $\Longrightarrow$ i). It is trivial.
\end{proof}
A separable Banach space $X$ is said to be  separably injective if it has the following extension property:
Every bounded linear operator from a closed subspace of a separable Banach space into $X$ can be extended to be a bounded operator on the whole space. In 1941, Sobczyk \cite {sob} showed that $c_0$ is separably
 injective, and  Zippin  (\cite{zip}, 1977) further
proved that $c_0$ is, up to isomorphism, the only separable separably injective space.

With the aim of Zippin's theorem, we can prove the following theorem, which says that $c_0$ is (up to isomorphism) the only space satisfying inequality \eqref{E:2.1} for every separable $Y$.

\begin{thm}\label{T:3.5} Let $X$ be a separable Banach space. Then the following statements are equivalent.

i) $(X,Y)$ is stable for every separable Banach space $Y$;

ii) $X$ is isomorphic to $c_0$;

iii) $(X,Y)$ is $(2\alpha,8\alpha)$-stable for every separable Banach space $Y$, where $\alpha=\|T\|\|T^{-1}\|$ for any isomorphism $T: X\rightarrow c_0$.

\end{thm}

\begin{proof} i) $\Longrightarrow$ ii).  Suppose that  $X$ is not  isomorphic to $c_0$. Then by Zippin's theorem, $X$ is not separably injective. Therefore, there exists a separable Banach space $Y$, which contains $X$ as an uncomplemented subspace. Clearly, $\text{card}(X)=\text{card}(Y)$. By Lemma \ref{L:2.1} again,
for every $\eps>0$, there is a standard $\eps$-isometry
 $f: X\rightarrow Y$ which is not stable.

ii) $\Longrightarrow$ iii).
 Let $X$ be a Banach space isomorphic to $c_0$ and $T: X\rightarrow c_0$ be an isomorphism.
 Assume that $(e_n)_{n=1}^\infty$ is the canonical basis of $c_0$ with the standard biorthogonal functionals $(e_n^*)_{n=1}^\infty\subset\ell_1.$
 Let $(x_n)\subset X$ satisfy $Tx_n=e_n$ for all $n\in\mathbb N$, and let $T^*:\ell_1\rightarrow X^*$ be the conjugate operator of $T$. Then
$$Tx=\sum (T^*e_n^*)(x)e_n\;{\rm and} \;x=\sum (T^*e_n^*)(x)T^{-1}e_n,\;\;{\rm for\;all}\;x\in X.$$

 Let $\alpha=\|T\|\cdot\|T^{-1}\|$,  $x_n^*= T^*e_n^*\in \|T\|B_{X^*}$ for all $n\in\mathbb N$, and note
 $x_n=T^{-1}e_n \in X$. By Theorem \ref{T:1.3} there exists
$\phi_{n}\in \|T\|B_{Y^*}$ with $\|\phi_{n}\|=\|x_n^*\|$ such that

\begin{align}\label{E:3.3}
|\langle\phi_{n},f(x)\rangle-\langle x_n^*,x\rangle|
\leq4 \eps\|T\|,\; {\rm for \;all}\; x\in X.\end{align}

Since $e^*_n\rightarrow0$ in the $w^*$-topology of $\ell_1=c_0^*$, $x^*_n=T^*e_n^*\rightarrow0$ in the $w^*$-topology of $X^*$. Let
\begin{align}K=\{\psi\in \|T\|B(Y^*):|\langle\psi, f(x)\rangle|\leq 4\eps\|T\|, \; {\rm for \;all}\; x\in X\}.\end{align}
Then $K$ is a nonempty $w^*$-closed compact subset of $Y^*$. Since $Y$ is separable, $(\|T\|B_{Y^*},w^*)$ is metrizable. Let $\rho$ be a metric such that $(\|T\|B_{Y^*},\rho)$ is isomorphic to $(\|T\|B_{Y^*},w^*)$.  Since $(\|T\|B_{Y^*},\rho)$ is a compact metric space and since $K$ is a compact subset of it, $(\phi_n)\subset K$ has at least one $\rho$-sequentially cluster point. Since $(x^*_n)$ is a $w^*$-null sequence in $X^*$, inequality \eqref{E:3.3} entails that any $\rho$-cluster point $\phi$ of $(\phi_n)$ is in $K$ and with $\|\phi\|\leq\|T^*\|=\|T\|$. This further implies that ${\rm dist}_\rho(\phi_n,K)\rightarrow0$. Consequently, there is a sequence $(\psi_n)\subset K$ such that dist$_\rho(\phi_n,\psi_n)\rightarrow0$, or equivalently, $\phi_n-\psi_n\rightarrow0$ in the $w^*$-topology of $Y^*$. Hence, for every $y\in Y$,
\begin{align}Uy\equiv\sum_{n=1}^\infty\langle\phi_n-\psi_n,y\rangle e_n\in c_0\end{align} and with \begin{align}\|Uy\|\leq(\sup_{n\in\mathbb N}\|\phi_n-\psi_n\|)\|y\|\leq2\|T\|\|y\|,\end{align}
 that is, $\|U\|\leq2\|T\|.$

Finally, let \begin{align}S(y)=T^{-1}(Uy)=\sum_{n=1}^\infty\langle \phi_n-\psi_n,y\rangle x_n \;{\rm for\; all}\; y\in Y.\end{align} Then
$$\|S\|=\|T^{-1}U\|\leq2\|T\|\cdot\|T^{-1}\|=2\alpha$$ and

\begin{align}&\|Sf(x)-x\|=\|\sum_{n=1}^\infty \langle\phi_n-\psi_n,f(x)\rangle x_n-\sum_{n=1}^\infty
 \langle x_n^*,x\rangle x_n\|\notag\\
&=\lim_{n\rightarrow\infty}\|\sum_{i=1}^n \langle\phi_i-\psi_i,f(x)\rangle x_i-\sum_{i=1}^n\langle x_i^*,x\rangle x_i\|\notag\\
&=\lim_{n\rightarrow\infty}\|\sum_{i=1}^n (\langle\phi_i,f(x)\rangle-\langle x_i^*,x\rangle) x_i-\sum_{i=1}^n\langle\psi_i,f(x)\rangle x_i\|\notag\\
&\leq\limsup_{n\rightarrow\infty}\|\sum_{i=1}^n (\langle\phi_i,f(x)-\langle x_i^*,x\rangle) x_i\|+\limsup_{n\rightarrow\infty}\|\sum_{i=1}^n\langle\psi_i,f(x)\rangle x_i\|\notag\\
&=\limsup_{n\rightarrow\infty}\|T^{-1}\sum_{i=1}^n (\langle\phi_i,f(x)\rangle-\langle x_i^*,x\rangle) e_i\|+\limsup_{n\rightarrow\infty}\|T^{-1}(\sum_{i=1}^n\langle\psi_i,f(x)\rangle e_i)\|\notag\\
&\leq\|T^{-1}\|\cdot\limsup_{n\rightarrow\infty}(\|\sum_{i=1}^n (\langle\phi_i,f(x)\rangle-\langle x_i^*,x\rangle) e_i\|+\|\sum_{i=1}^n\langle\psi_i,f(x)\rangle e_i)\|)\notag\\
&\leq\|T^{-1}\|(\sup_n |\langle\phi_i,f(x)\rangle-\langle x_i^*,x\rangle| +\sup_n|\langle\psi_i,f(x)\rangle|)\notag\\
&\leq 8 \eps\|T\|\cdot\|T^{-1}\|=8\eps\alpha.\notag\end{align}
Thus, our proof is complete.
\end{proof}

\subsection*{Acknowledgements}
The authors would like to thank the referee for his (her) insightful and helpful suggestions on this paper.

\end{document}